\documentclass[11pt,a4paper]{article}
\usepackage{amsmath}

\usepackage{epsf,epsfig,amsfonts,amsgen,amsmath,amstext,amsbsy,amsopn,amsthm
}
\usepackage{ebezier,eepic}
\usepackage{color}
\usepackage{multirow}
\usepackage{cite}
\allowdisplaybreaks[2]

\newtheorem{thm}{Theorem}

\newtheorem{lem}{Lemma}
\newtheorem{false statement}{False statement}

\theoremstyle{definition}

\newtheorem{claim}{Claim}

\newcounter{mathitem}
  {\begin{list}{{$(\roman{mathitem})$}}{
   \setcounter{mathitem}{0}
   \usecounter{mathitem}
   \setlength{\topsep}{0pt plus 2pt minus 0pt}
   \setlength{\parskip}{0pt plus 2pt minus 0pt}
   \setlength{\partopsep}{0pt plus 2pt minus 0pt}
   \setlength{\parsep}{0pt plus 2pt minus 0pt}
   \setlength{\leftmargin}{35pt}
   \setlength{\itemsep}{0pt plus 2pt minus 0pt}}}
  {\end{list}}

\title{On degree power sum in $P_k$-free graphs}
\author{Jiangdong Ai\thanks{Corresponding author.  School of Mathematical Sciences and LPMC, Nankai University, Tianjin 300071, P.R.
China. Email: jd@nankai.edu.cn. },~ Fankang He\thanks{School of Mathematical Sciences and LPMC, Nankai University, Tianjin 300071, P.R.
China. Email: hefankang@mail.nankai.edu.cn.},~ Yihang Liu\thanks{School of Mathematical Sciences and LPMC, Nankai University, Tianjin 300071, P.R.
China. Email: liuyihang@mail.nankai.edu.cn.},~Bo
Ning\thanks{College of Computer Science, Nankai University, Tianjin 300350, P.R.
China. Email: bo.ning@nankai.edu.cn. Partially supported
by the NSFC grants (Nos.\ 12371350 and 11971346).}}

\date{}

\begin{document}
\maketitle
\begin{abstract}
Let $G$ be a graph on $n$ vertices with degree sequence $(d_1,d_2......d_n)$. For a real $p \geq 1$, let $D_p(G)=\sum_{i=1}^nd_i^p$. A Tur\'an-type problem of degree power sum was initiated by Caro and Yuster \cite{caro2000degpower}: determining the function $D_p(n,H) :=\max \{D_p(G): \text{$G$ is an $n$-vertex $H$-free graph}\}$. They obtained some exact values for certain graphs $H$. For a path $P_k$, they mentioned that ``a close examination of the proof of Theorem 1.2 shows that the value of $n_0(k)$ in the statement of the theorem is $O(k^2)$", namely, they could show the $n$-vertex $P_k$-free graph with maximum degree power sum is $W_{n,k-1,\lfloor \frac{k}{2} \rfloor -1} = K_{\lfloor \frac{k}{2} \rfloor -1} \vee \left((n - \lceil \frac{k}{2} \rceil)K_1 \cup K_{1+k-2\lfloor \frac{k}{2} \rfloor} \right)$ when $n \geq c k^2$ for some constant $c$. In this note, we improve their result to a linear size of $k$ by a different approach. The bound is tight up to a constant factor.
\end{abstract}
\maketitle
\section{Introduction}
Let $G$ be a graph on $n$ vertices whose degree sequence is $(d_1,d_2,\ldots,d_n)$. For a real $p \geq 1$, we denote the \textit{degree power sum} of $G$ by $D_p(G) :=\sum_{i=1}^nd_i^p$. This parameter, which is also known as the \emph{general zeroth-order Randi\'c index}, is well-studied in chemical graph theory, see the comprehensive survey \cite{ali2018survey}. Caro and Yuster\cite{caro2000degpower} initiated a Tur\'an-type problem of degree power sum: determining the function $D_p(n,H) :=\max \{D_p(G): \text{$G$ is an $n$-vertex $H$-free graph}\}$. Since then, a lot of research on this Tur\'an-type problem has appeared, see \cite{caro2000degpower, bollobas2004degree, pikhurko2005degree, bollobas2012degree, gu2015degree, lan2019degree, zhang2022degree, chang2023extremal}. 

In this note, we focus on this Tur\'an-type problem when the forbidden graph is a path. Caro and Yuster \cite{caro2000degpower} determined $D_p(n,P_k)$ for any integers $k\geq 2$, $p \geq 2$ and sufficiently large $n$,  and characterized the corresponding extremal graphs. Define $W_{n,k,s}$ to be a graph on $n$ vertices, in which its vertex set can be partitioned into three subsets $X,Y,Z$ such that $|X|=s$, $|Y|=k-2s$, $|Z|=n-(k-s)$, and the edge set consists of all possible edges between $X$ and $Z$ and all edges in $X\cup Y$. 

\begin{thm}[\cite{caro2000degpower}]\label{thm:degpwpath}
For two integers $k \geq 4$ and $p \geq 2$, there exists a positive integer $n_0 = n_0(k)$ such that for all $n \geq n_0$, the following holds. If $G$ is an $n$-vertex $P_k$-free graph with maximum degree power sum $D_p(G)$, then $G$ is $W_{n,k-1,t}$ where $t = \lfloor \frac{k}{2} \rfloor - 1$.
\end{thm}

Note that Caro and Yuster obtained the above result for integers $p\geq 2$. A natural question arising from Theorem~\ref{thm:degpwpath} is to determine, for fixed real $p$, how large $n_0(k)$ should be as a function of $k$. Caro and Yuster mentioned that ``a close examination of the proof of Theorem 1.2 shows that the value of $n_0(k)$ in the statement of the theorem is $O(k^2)$" (quoted from \cite{caro2000degpower}). In this note, by a different method, we prove that $n_0(k)$ can be as small as $10k$. And our result extends Theorem~\ref{thm:degpwpath} to any real $p \geq 2$. We remark that the bound is tight up to a constant factor. Note that $K_{k-1}$, rather than $W_{k-1, k-1, t}$ where $t = \lfloor \frac{k}{2} \rfloor - 1$, attains the maximum degree power sum among all $(k-1)$-vertex $P_k$-free graphs. This yields $n_0(k) \geq k$.

\begin{thm}\label{thm:mainthm}
For an integer $k \geq 4$ and real $p \geq 2$, let $G$ be an $n$-vertex $P_{k}$-free graph with the maximum degree power sum $D_p(G)$. If $n \geq 10k$, then $G$ is $W_{n,k-1,t}$ where $t = \lfloor \frac{k}{2} \rfloor - 1$.
\end{thm}

Let us outline the main ideas behind the proof of Theorem~\ref{thm:mainthm}. Given an $n$-vertex $P_k$-free graph $G$ with the largest degree power sum, we will first show that there are a few components. This will allow us to use structural information to guarantee the existence of a large component $G_i$. Next, we will show that if there is any other component $G_j$, then we can use $G_i$ to ``absorb" $G_j$ to get a larger degree power sum which contradicts the maximality of the degree power sum of $G$.

\section{Proofs}
Recently, Ai, Lei, Ning, and Shi \cite{ai2023graph} proved two connected versions of Theorem \ref{thm:degpwpath}. Based on these results, we characterize some structures in the proof of Theorem~\ref{thm:mainthm}.

\begin{lem}[\cite{ai2023graph}]\label{lem:conpath}
    Let $k \geq 4$ and $t = \lfloor \frac{k}{2} \rfloor - 1$ and let $p \geq 2$. The following holds for all $n \geq k$. If $G$ is a connected $n$-vertex $P_k$-free graph with maximum degree power sum $D_p(G)$, then $G$ is $W_{n,k-1,1}$ or $W_{n,k-1,t}$.
\end{lem}

\begin{lem}[\cite{ai2023graph}]\label{lem:condegpwpath}
    Let $k \geq 4$ and $t = \lfloor \frac{k}{2} \rfloor - 1$ and let $p \geq 2$. The following holds for all $n \geq 2k$. If $G$ is a connected $n$-vertex $P_k$-free graph with maximum degree power sum $D_p(G)$, then $G$ is $W_{n,k-1,t}$.
\end{lem}

The following lemma is useful in the proof of the main theorem.

\begin{lem}\label{lemma:inequality}
Let $a \geq b >0$ and $p \geq 2$. Then $(a+b)^p \geq a^p + pa^{p-1}b + b^p$. In particular, $(a+b)^p \geq a^p + (p+1)b^p$.
\end{lem}

\begin{proof}
We first show that $(x+1)^p \geq x^p + px^{p-1} + 1$ when $x \geq 1$. Let $f(x) = (x+1)^p - x^p - px^{p-1} - 1$ and so $f(1) = 2^p - p - 2 \geq 0$. Taking derivative of $f(x)$, we get $f'(x) = p\left((x+1)^{p-1} - x^{p-1} - (p-1)x^{p-2}\right)$. Note that $f'(x) \geq 0$ due to convexity of $x^{p-1}$. Hence, $f(x) \geq 0$ when $x \geq 1$ as desired. 
Then
    \[
        (a+b)^p = b^p \left(\frac{a}{b} + 1\right)^p \geq b^p \left(\left(\frac{a}{b}\right)^p + p\left(\frac{a}{b}\right)^{p-1} + 1\right) = a^p + pa^{p-1}b + b^p \geq a^p + (p+1)b^p.
    \]
    This completes the proof.
\end{proof}

\begin{proof}[Proof of Theorem~\ref{thm:mainthm}]
Let $G$ be an $n$-vertex $P_k$-free graph which attains the maximum degree power sum $D_p(G)$ where $p\geq 2$. Let $t=\lfloor \frac{k}{2} \rfloor - 1$ and let $G_1,G_2,\ldots,G_{\ell}$ be all components of $G$. Let $n_i = |G_i|$ for $i \in [\ell]$. By {Lemma~\ref{lem:conpath}}, we can immediately get the following claim.

\setcounter{claim}{0}
\begin{claim}\label{claim:extremal}
For any component $G_i$, where $1\leq i\leq \ell$, if $n_i \leq k-1$, then $G_i$ is a clique; if $n_i \geq k$, then $G_i \in \{W_{n_i,k-1,1}, W_{n_i,k-1,t}\}$. 
\end{claim}

\begin{proof}
    When $n_i \leq k-1$, then any $n_i$-vertex graph is $P_k$-free. Hence, $K_{n_i}$ would attain the largest degree power sum among all $n_i$-vertex graphs. When $n_i \geq k$, by Lemma~\ref{lem:conpath}, we have $G_i \in \{W_{n_i,k-1,1}, W_{n_i,k-1,t}\}$. 
\end{proof}

In the rest of this proof, we say that a component $G_i$ is a \emph{Type-A} component if $G_i = W_{n_i, k-1,1}$, and is a \emph{Type-B} component if $G_i = W_{n_i, k-1,t}$ where $t = \lfloor \frac{k}{2} \rfloor - 1$.

\begin{claim}\label{claim:twosize}
For any two components $G_1$
and $G_2$ of order $n_1\leq n_2\leq k-1$, we have $n_1+n_2\geq k$.
\end{claim}

\begin{proof}
Suppose that $n_1+n_2\leq k-1$.
Then $G_i$ must be a clique for $i=1,2$. Note that $D_p(G_1\cup G_2)=n_1(n_1-1)^p+n_2(n_2-1)^p$. If we replace $G_1\cup G_2$ with $K_{n_1 + n_2}$ and denote the new graph by $G'$. Then $G'$ is still $P_k$-free since $n_1 + n_2 \leq k-1$. And $D_p(K_{n_1 + n_2})=(n_1+n_2)(n_1+n_2-1)^p$.
Note that $D_p(G_1\cup G_2)=n_1(n_1-1)^p+n_2(n_2-1)^p \leq (n_1+n_2-1)((n_1-1)^p+(n_2-1)^p)<D_p(K_{n_1 + n_2})$, as there holds the inequality $x^p+y^p < (x+y+1)^p$ where $x\geq 0,y\geq 0$ and $p\geq 2$. Now, we proved $G'$ has a larger degree power sum than $G$, a contradiction. This proves the claim.
\end{proof}

\begin{claim}\label{claim:k-2}
    There exists at most one component of order at most $k-2$.
\end{claim}

\begin{proof}
    Suppose to the contrary that there exist two components, say $G_1$ and $G_2$, of order $n_1 \leq k-2$ and $n_2 \leq k-2$, respectively. By Claim~\ref{claim:twosize}, $n_1 + n_2 \geq k$. Let $n'_1 = k-1$ and $n'_2 = n_1 + n_2 - k+1$. Note that $n'_2 \leq k-3$, then $K_{n'_1} \cup K_{n'_2}$ contains no $P_k$. Recall that $D_p(K_r) = r(r-1)^p$. Then $D_p(K_{n'_1} \cup K_{n'_2}) > D_p(G_1 \cup G_2)$ by convexity argument. This means that if we replace $G_1 \cup G_2$ with $K_{n'_1} \cup K_{n'_2}$, then we get an $n$-vertex $P_k$-free graph with a larger degree power sum, a contradiction.
\end{proof}

\begin{claim}\label{claim:k-1}
    There exist at most five components of order $k-1$. 
\end{claim}

\begin{proof}
    Suppose to the contrary that there exist six components, say $G_i, i=[6]$ of order $n_i = k-1, i \in [6]$. Then 
    \begin{align*}
        D_p(W_{6k-6, k-1, t}) - D_p(\bigcup_{i=1}^6 G_i) =& \left(\lfloor \frac{k}{2} \rfloor - 1\right)(6k-7)^p + (6k-\lceil \frac{k}{2} \rceil -6)\left(\lfloor \frac{k}{2} \rfloor - 1\right)^p \\
        &+ \left(1+k - 2\lfloor \frac{k}{2} \rfloor\right)\left(\lceil \frac{k}{2} \rceil - 1\right)^p - 6(k-1)(k-2)^p \\
        >& 6^p \left(\lfloor \frac{k}{2} \rfloor - 1\right)\left(k-\frac{7}{6}\right)^p - 6(k-1)(k-2)^p \\
        \geq& 6^2 \left(\frac{k}{2} - \frac{3}{2}\right)(k-2)^p - 6(k-1)(k-2)^p \\
        =& (12k - 48)(k-2)^p \geq 0.
    \end{align*}
    This means that we can replace $\bigcup_{i=1}^6 G_i$ with $W_{6k-6, k-1, t}$ to get an $n$-vertex $P_k$-free graph with a larger degree power sum, a contradiction.
\end{proof}

\begin{claim}\label{claim:type-A}
    There exists at most one Type-A component.
\end{claim}

\begin{proof}
    Suppose to the contrary that there exist two Type-A components, say $G_1$ and $G_2$, of order $n_1$ and $n_2$, respectively. Then by Claim~\ref{claim:extremal} and Lemma~\ref{lem:condegpwpath}, $k \leq n_1 \leq 2k-1$ and $k \leq n_2 \leq 2k-1$. 
    Let $n'_1 = k-1$ and $n'_2 = n_1 + n_2 - k+1$, let $G'_1 = W_{n'_1,k-1,1}$ and $G'_2 = W_{n'_2, k-1, 1}$. Recall that $D_p(W_{n,k-1,t}) = (n-k+2) + (k-3)^p + (n-1)^p$ holds for $n \geq k-2$. Then $D_p(G'_1 \cup G'_2) > D_p(G_1 \cup G_2)$ by convexity argument. Now we can replace $G_1 \cup G_2$ with $G'_1 \cup G'_2$ to get an $n$-vertex $P_k$-free graph with a larger degree power sum, a contradiction.
\end{proof}

\begin{claim}\label{claim:type-B}
    There exists at most one Type-B component.
\end{claim}

\begin{proof}
    Suppose to the contrary that there exist two Type-B components, say $G_1$ and $G_2$, of order $n_1$ and $n_2$, respectively. W.O.L.G., suppose $n_1 \geq n_2$. By Claim~\ref{claim:extremal}, $n_1 \geq k$ and $n_2 \geq k$. Then $D_p(G_i) = (\lfloor \frac{k}{2} \rfloor - 1)(n_i - 1)^p + (n_i - \lceil \frac{k}{2} \rceil)(\lfloor \frac{k}{2} \rfloor - 1)^p + (1+k - 2\lfloor \frac{k}{2} \rfloor)(\lceil \frac{k}{2} \rceil - 1)^p, i = 1,2$. Let $G_0 = W_{n_1 + n_2, k-1, t}$. Note that
    \begin{align}
        D_p(G_0) - D_p(G_1 \cup G_2) =&\left (\lfloor \frac{k}{2}\rfloor - 1 \right) \left((n_1 + n_2 - 1)^p - (n_1 - 1)^p - (n_2 - 1)^p\right) \notag \\
        &+ \lceil \frac{k}{2} \rceil \left(\lfloor \frac{k}{2} \rfloor - 1\right) ^p - \left(1+k - 2\lfloor \frac{k}{2} \rfloor\right)\left(\lceil \frac{k}{2} \rceil - 1\right)^p \notag \\
        >& \left(\lfloor \frac{k}{2} \rfloor - 1\right ) \left((n_1 + n_2 - 2)^p - (n_1 - 1)^p - (n_2 - 1)^p\right) \notag \\
        &+ \lceil \frac{k}{2} \rceil \left(\lfloor \frac{k}{2} \rfloor - 1\right)^p - \left(1+k - 2\lfloor \frac{k}{2} \rfloor\right)\left(\lceil \frac{k}{2} \rceil - 1\right)^p \notag \\
        \geq& \left(\lfloor \frac{k}{2} \rfloor - 1\right) p(n_2 - 1)^p +  \lceil \frac{k}{2} \rceil \left(\lfloor \frac{k}{2} \rfloor - 1\right)^p \notag \\ 
        &- \left(1+k - 2\lfloor \frac{k}{2} \rfloor\right)\left(\lceil \frac{k}{2} \rceil - 1\right)^p \label{inq:1} \\
        \geq& 2(k-1)^p -2\left(\lceil \frac{k}{2} \rceil - 1\right)^p + \lceil \frac{k}{2} \rceil \left(\lfloor \frac{k}{2} \rfloor - 1\right)^p \notag \\ 
        >& 0. \notag
    \end{align}
    The inequality (\ref{inq:1}) holds due to Lemma~\ref{lemma:inequality} by letting $a = n_1 -1$ and $b = n_2 - 1$. Then we can replace $G_1 \cup G_2$ with $G_0$ to get an $n$-vertex $P_k$-free graph with a larger degree power sum, a contradiction.
\end{proof}

\begin{claim}\label{claim:BabsorbA}
    If $n_i \geq 2k, n_j \leq k-1$ and $G_i$ is of Type-B, then we can replace $G_i \cup G_j$ with $W_{n_i + n_j, k-1, t}$ to get a larger degree power sum.
\end{claim}

\begin{proof}
    Let $G_0 = W_{n_i + n_j, k-1, t}$. Then
    \begin{align}
        D_p(G_0) - D_p(G_i \cup G_j) =& \left(\lfloor \frac{k}{2} \rfloor - 1\right)\left((n_i + n_j - 1)^p - (n_i - 1)^p\right) \notag \\ 
        &+ n_j\left(\lfloor \frac{k}{2} \rfloor - 1\right)^p - n_j(n_j - 1)^p \notag \\
        \geq& \left(\lfloor \frac{k}{2} \rfloor - 1\right) p(n_i - 1)^{p-1}n_j + n_j\left(\lfloor \frac{k}{2} \rfloor - 1\right)^p - n_j(n_j - 1)^p \label{inq:2} \\
        \geq& p\left(\frac{k}{2} - \frac{3}{2}\right)n_j (2k - 1)^{p-1} - n_j(n_j - 1)^p + n_j\left(\lfloor \frac{k}{2} \rfloor - 1\right)^p \notag \\
        \geq& n_j (2k-6) \left(k - \frac{1}{2}\right)^{p-1}-n_j(k-2)^p+n_j\left(\lfloor \frac{k}{2} \rfloor - 1\right)^p \notag \\
        >& 0. \notag
    \end{align}
    The inequality (\ref{inq:2}) holds due to Lemma~\ref{lemma:inequality} by letting $a = n_i -1$ and $b = n_j$. This completes the proof.
\end{proof}

\begin{claim}\label{claim:Babsorbclique}
    If $n_i \geq 2k$, $G_i$ is of Type-B, and $G_j$ is of Type-A, then we can replace $G_i \cup G_j$ with $W_{n_i + n_j, k-1, t}$ to get a larger degree power sum.
\end{claim}

\begin{proof}
    Let $G_0 = W_{n_i + n_j, k-1, t}$. By Lemma~\ref{lem:condegpwpath} and Claim~\ref{claim:extremal}, we have $k \leq n_j \leq 2k-1$. Then
    \begin{align}
        D_p(G_0) - D_p(G_1 \cup G_2) =& \left(\lfloor \frac{k}{2} \rfloor - 1\right)\left((n_i + n_j - 1)^p - (n_i - 1)^p\right) + n_j\left(\lfloor \frac{k}{2} \rfloor - 1\right)^p \notag \\
        &-(n_j - k + 2) - (k-3)^{p+1} - (n_j - 1)^p \notag \\
        >& \left(\lfloor \frac{k}{2} \rfloor - 1\right)p(n_i - 1)^{p-1} n_j + n_j\left(\lfloor \frac{k}{2} \rfloor - 1\right)^p \notag \\
        &-(n_j - k + 2) - (k-3)^{p+1} - (n_j - 1)^p \label{inq:3} \\
        \geq& \left(\lfloor \frac{k}{2} \rfloor - 1\right) n_j (n_i - 1)^{p-1} - (k-3)^{p+1} \notag \\
        &+ \left(\lfloor \frac{k}{2} \rfloor - 1\right) n_j (n_i - 1)^{p-1} -(n_j - 1)^p - (n_j - k + 2) \notag \\
        \geq&  \left(\frac{k}{2} - \frac{3}{2}\right) k (2k - 1)^{p-1} - (k-3)^{p+1} \notag \\
        &+ n_j (n_i - 1)^{p-1} -(n_j - 1)^p - (n_j - k + 2) \notag \\
        \geq&  \left( k-3 \right) k \left(k - \frac{1}{2} \right)^{p-1} - (k-3)^{p+1} \notag \\
        &+ n_j (n_j - 1)^{p-1} -(n_j - 1)^p - (n_j - k + 2) \notag \\
        >& 0 \notag.
    \end{align}
    The inequality (\ref{inq:3}) holds due to Lemma~\ref{lemma:inequality} by letting $a = n_i -1$ and $b = n_j$. This completes the proof.
\end{proof}

\begin{claim}
    There is no component of order at most $k-1$. And there is no Type-A component.
\end{claim}

\begin{proof}
We first show that there is a component of order at least $2k$. By Claim~\ref{claim:k-2} and Claim~\ref{claim:k-1}, there are at most six components of order at most $k-1$. By Claim~\ref{claim:type-A}, there is at most one Type-A component. Lemma~\ref{lem:condegpwpath} tells us that the order of a Type-A component is at most $2k-1$. Combined with the assumption that $n\geq 10k$, there is a component of order at least $10k-6(k-1)-2k+1=2k+7$, which is of Type-B.

If there is a component of order at most $k-1$, then by Claim~\ref{claim:BabsorbA}, we can get an $n$-vertex $P_k$-free graph with a larger degree power sum. If there is a Type-A component, then by Claim~\ref{claim:Babsorbclique}, we can get an $n$-vertex $P_k$-free graph with a larger degree power sum. Now, we are done.
\end{proof}

As a result, $G$ has only one component which is of Type-B, i.e., $G$ is $W_{n,k-1,t}$. 
This proves the theorem.
\end{proof}

\bibliographystyle{abbrv}
\bibliography{ref.bib}
\label{key}  

\end{document}